\RequirePackage{fix-cm}
\documentclass[smallextended]{svjour3}       
\smartqed  
\usepackage{graphicx}
\usepackage[all]{xy}
\begin{document}

\title{The catenary degree of the saturated numerical semigroups with prime multiplicity
}
\author{Meral S\"{u}er
}


\institute{M. S\"{u}er \at
              Department of Mathematics, Faculty of Science and Letters, Batman University, 72100 Batman, Turkey \\
              \email{meral.suer@batman.edu.tr}  
}

\date{Received: date / Accepted: date}

\maketitle

\begin{abstract}
In this paper we present the set of saturated numerical semigroups with prime multiplicity. We also characterize
the catenary degree of these semigroups that we acquire. The catenary degree of a numerical semigroup is the
variant which measures the distance between factorizations of elements within that numerical semigroup.
\end{abstract}

\section{Introduction}
\label{intro}
Researchers have been interested in two different aspects of non-unique factorization invariants. One of them includes the ones based on the lengths of the factorizations of an element, while the other discusses the uses the idea of distance between factorizations. In the first case, in a half-factorial monoid with all factorizations of the same length of a given element, only the semi-factor property is dealt with, while in the other case, focus is on the catenary and tame degree. We will deal with the second case in this study. An element of a cancellative monoid is expressed in different ways as a linear combination with non-negative integer coefficients of its generators. This expression is known as a factorization of that element. The catenary degree of the element of the cancellative monoid is combinatorial constant that describe the relationships between differing irreducible factorizations of the element. The supremum of all catenary degrees of all the elements in the monoid is the catenary degree of the monoid itself.
\paragraph{}In the past 20 years, problems involving non-unique factorizations of elements in integral domains and commutative cancellative monoids have become very popular in the mathematical literature (\cite{Ref16} and its citation list). Most of this studies concentrate on various combinatorial constants which describe, in a sense, how these systems differ from the classical notion of unique factorization. The earliest studies in this area are on Krull domains and monoids \cite{Ref3,Ref5,Ref10,Ref11,Ref14,Ref15,Ref17,Ref21}. Recent studies in this area evaluate these properties on numerical monoid \cite{Ref1,Ref5,Ref6,Ref7,Ref8,Ref12,Ref19,Ref20}.
\paragraph{}In the literature, a long list of studies can be found on the analysis of one-dimensional analytically irreducible local domains via value semigroups \cite{Ref4}. One of the properties studied for this kind of ring using aforementioned approach is the saturated rings. Definitions of saturated rings correspond to algebraically closed fields of zero characteristic. Saturated numerical semigroups come up after a characterization of saturated rings in terms of their value semigroup \cite{Ref9,Ref18}. Although the concept of saturated semigroups is included in the theory of the ring, it first attracted the attention of semigroupist \cite{Ref26,Ref29,Ref30}.
\paragraph{}The paper is organized as follows. In Sect.~\ref{sec:2} we will include the necessary definitions and notations that we will use our main result and proofs. In Sect.~\ref{sec:3} (Theorem \ref{thm1}) we will find all saturated numerical semigroup with prime multiplicity and fixed conductor. Finally, in Sect.~\ref{sec:4}(Theorem \ref{thm2} and Theorem \ref{thm3}) we will obtain the catenary degree of the saturated numerical semigroup with prime multiplicity and fixed conductor.
\section{Definitions and preliminaries}
\label{sec:2}
Let $\mathbf{Z}$ and $\mathbf{N}$ be the set of integers and non-negative integers, respectively. A numerical semigroup is a subset $S$ of $\mathbf{N}$ that is closed under addition, $0 \in S$ and $\mathbf{N}\setminus S$ has finitely many elements. The set $\mathbf{Z}\setminus S$ has a maximum, which is known as the Frobenius number of S, denoted by $F(S)$ \cite{Ref22}. The least integer $s $ that provides $s+n\in S $  for all $n\in \mathbf{N}$ is called the conductor of $S$, denoted here by $c(S)$ (in short $c$). c is actually the Frobenius number of S plus one \cite{Ref4}.
\paragraph{} Given $A$ a nonempty subset of $\mathbf{N}$, $\langle A\rangle$ denotes the submonoid of $(\mathbf{N}, +)$ generated by A, that is 
$$ \langle A\rangle= \left\lbrace n_1a_1 +\dots+ n_ra_r :r\in \mathbf{N}\setminus \{0\}, n_1, \dots.n_r \in \mathbf{N},  a_1, \dots.a_r \in A\right\rbrace$$
If $S= \langle A\rangle, A$ is a system of generators of $S$. In this case, we say that $A$ is a minimal system of generators of $S$ if no proper subset of $A$ generates $S$. It is well known that every numerical semigroup admits a unique minimal system of generators, which has finitely many elements \cite{Ref4,Ref27}. It is also well known that $S=\langle A\rangle$ is a numerical semigroup if and only if $gcd(A)=1$, where $gcd$  stands for greatest common divisor\cite{Ref28}. If $S$ is a numerical semigroup and its minimal system of generators is $A= \left\lbrace a_1<a_2<\dots<a_r\right\rbrace$, then $a_1,a_2$ and $r$ called the multiplicity,the ratio and the the embedding dimension of S, these are denoted by $\mu(S),R(S)$ and $e(S)$, respectively. It is known that $e(S)\leq\mu(S)$. If $S$ is a numerical semigroup with embedding dimension that is equal to multiplicity, it has maximal embedding dimension. The numerical semigroups with maximal embedding dimension are denoted by MED-semi groups for short.
For a numerical semigroup $S$ and $s \in S\setminus\{0\} $, the  Ap\'ery set of $s$ in $S$ is defined by
$$ Ap(S,s)= \left\lbrace x\in S : x-s\neq S \right\rbrace$$ 
It is well known (see for instance \cite{Ref28}) that
$$ Ap(S,s)=\left\lbrace w_0=0,w_1,\dots,w_{s-1} \right\rbrace$$ 
$w_i=\min \left\lbrace x\in S: x \equiv i(mods)\right\rbrace$  for $i=\left\lbrace 0,1,\dots,s-1\right\rbrace $. Readers can see the following definitions and results in more detail in \cite{Ref2,Ref28}. A numerical semigroup $S$ is called Arf if $x + y-z\in S$ for all $x, y, z \in S$ where $x\leq y \leq z$. 
\paragraph{} A numerical semigroup S is saturated if the following condition holds: if $s,s_1,\dots,s_r\in  S$ are such that $s_r\leq s$ for all $i\in \{1,\dots,r\}$ and $z_1,\dots,z_r\in \mathbf{Z}$ are such that $s_1 z_1 +\dots+ s_r z_r \geq 0$, then $s+s_1 z_1 +\dots+ s_r z_r\in S$.
For $A$ a nonempty subset of $\mathbf{N}$ and $a\in A\setminus \left\lbrace 0 \right\rbrace$, the set
$$ d_A (a)=gcd \left\lbrace x\in A:x\leq a\right\rbrace $$ 
$A$ numerical semigroup $S$ is called saturated if $s+d_S(s)\in S$ for all $s\in S\setminus \left\lbrace 0 \right\rbrace$. It is well known that any saturated numerical semigroup has the Arf property, whence it is of maximal embedding dimension \cite{Ref4,Ref9}.
\paragraph{}  From the definition of saturated numerical semigroup, it is easy to deduce that giving a saturated numerical semigroup $S$ is equivalent to give a sequence of positive integers  $s_1<s_2<\dots <s_r$ with greatest common divisor one and $ gcd \left\lbrace s_1,s_2,\dots,s_i\right\rbrace\neq gcd \left\lbrace s_1,s_2,\dots,s_i,s_{i+1}\right\rbrace $ for all $i\in \left\lbrace 1,2,\dots,r-1\right\rbrace$. In this case, we say
that $\left\lbrace  s_1,s_2,\dots,s_r\right\rbrace $ is a minimal SAT-system of generators of $S$. Furthermore, if $ d_i=gcd \left\lbrace s_1,s_2,\dots,s_i\right\rbrace $  for each  $i\in \left\lbrace 1,\dots,r\right\rbrace$, then we say that $S$ is a  $\left( s_1,s_2,\dots ,s_r\right) -$semigroup. A saturated sequence of length $k$, is a $k-$tuple of positive integers  $\left( d_1,d_2,\dots ,d_k\right)$ such that $ d_1>d_2>\dots >d_k=1$ and $d_{i+1}\mid d_i$ for all $i\in \left\lbrace 1,\dots,k-1\right\rbrace$.
Let  $F$ be positive integer. An $F-$ saturated sequence is a saturated sequence $\left( d_1,d_2,\dots ,d_k\right)$ such that there exists at least one $\left( d_1,d_2,\dots ,d_k\right)-$semigroup with Frobenius number $F$ \cite{Ref26}. 
\paragraph{}  $S$ be a numerical semigroup minimally generated by $ \{a_1,…,a_r \}$. The homomorphism
$$ \varphi :\mathbf{N}^r \rightarrow S,\varphi (a_1,\dots,a_r )=n_1a_1+\dots+n_r a_r $$
is the factorization homomorphism of S.The monoid $S$ is isomorphic to $\mathbf{N}^r / \sigma$,
where $a\sigma b$ if $\varphi(a)=\varphi(b)$. The congruence $\sigma$ is the kernel congruence of $\varphi$. The set of factorizations of $s$ in $S$ is
$$Z(s)=\varphi^{-1} (s)=\left\lbrace (n_1,…,n_r )\in \mathbf{N}^r:n_1a_1+\dots+n_r a_r=s \right\rbrace.$$
For a factorization $x=(x_1,…,x_r )\in Z(s)$, its length is
$$ |x|=x_1+\dots+x_r, $$
and the set of lengths of factorization of s is
$$ L(s)=\left\lbrace |x|:x\in Z(s)\right\rbrace=\left\lbrace m_1,…,m_l \right\rbrace. $$
The set of lengths of factorization of an element in a numerical semigroup is finite. Furthermore, if  $S=\mathbf{N}$, then there will always be elements with more than one length.
Let $ x=(x_1,\dots,x_r ),y=(y_1,\dots,y_r )\in \mathbf{N}^r $ be two factorization and let
$$ gcd(x,y) =(\min\left\lbrace x_1,y_1\right\rbrace,\dots,\min\left\lbrace x_r  ,y_r  \right\rbrace) $$
be their common part. The distance between $x$  and $y$ is
$$ dist(x,y) =\max\left\lbrace |x -gcd(x,y)|,|y -gcd(x,y)|\right\rbrace =\max\left\lbrace |x|,|y|\right\rbrace-gcd(x,y). $$
The support of  $x\in \mathbf{N}^r$ is defined by
$$ supp(x)=\left\lbrace  i : x_i\neq 0,1\leq i \leq r\right\rbrace $$
Let $s\in S$ be such that $s-s_i\in S$. Then the set
$$Z^i(s)=\left\lbrace x\in Z(s):i\in supp(x)\right\rbrace $$
is not empty.
Let $N\in \mathbf{N}$. A finite sequence $z=z_0,z_1,\dots,z_{n-1},z_n$ of factorization of $s\in S$ is an $N-$  chain if $dist(z_{n-1},z_i)\leq N$ for each $1\leq i \leq n$. We define the catenary degree of $s$ (denoted by $C(s)$) to be the minimal $N$ such that there is an $N-$ chain between any two factorization of s. The catenary degree of $S$, denoted by $C(S)$, is
$$ C(S) = sup\left\lbrace C(s) | s\in S\right\rbrace. $$
\paragraph{} $A$ presentation for $S$ is a subset $\rho$ of $\sigma$ such that $\sigma$  is the least congruence (withrespect to set inclusion) containing $\rho$. That is, a system of generators of $\sigma$. Every finitely generated commutative monoid is finitely presented, and thus every numerical semigroup is finitely presented \cite{Ref23}. Moreover, for numerical semigroups the concepts of minimality with respect to cardinality and set inclusion of a presentation coincide.
Two elements $a$ and $b$ in $\mathbf{N}^r$ are R-related if there exist $a$ chain $a=z_0,z_1,\dots,z_{n-1},z_n=b$ such that  $supp(z_(i-1) )\bigcap  supp(z_i )$ is not empty for $1\leq i\leq n$.This is an equivalence  binary relation on $Z(s)$ for $s\in S$. Since the number of factorization of an element in a numerical semigroup is finite, the number of class $\Re-$ classes in this set is also finite. $\Re-$ classes  are crucial, since from them a minimal presentation of $S$ can be constructed. Let  $s\in S$ and let $\Re_1^s,\dots,\Re_{n_s}^s$ be the different $\Re-$ classes of $Z(s)$. Set $m(s)=\max\left\lbrace r_1^s,\dots,r_{n_s}^s \right\rbrace $ where $r_i^s=\min\left\lbrace  |z|:z\in  R_i^s \right\rbrace$. Denote by $m(S)= \max\left\lbrace m(s): s\in S \quad \textrm{and} \quad n_s\geq2\right\rbrace$. It is known that $C(S) = m(S)$[8].
\paragraph{} For $A,B\subset\mathbf{N}$, we set 
$$ A+B=\left\lbrace a+b: a\in A,b\in B\right\rbrace,\quad nA= \underbrace{A+A+\dots+A}_{n}.$$
\section{The Saturated Numerical Semigroups with Prime Mutiplicity}
\label{sec:3}
In this section we are interested in calculating the set of all saturated numerical semigroups
with prime multiplicity and fixed conductor.
\begin{lemma} \label{lemma1} {\cite{Ref32}}
Let $S$ be a numerical semigroup with minimal system of generators
$ a_1<a_2<\dots<a_e$ and let $x \in S\setminus \{0\} $. Then
 \begin{enumerate}
	\item [i.] $\sharp Ap(S,x)=x $ ($\sharp$ stands for cardinality),
	\item [ii.]  $F(S)=\max (Ap(S,x))-x$ 
	\item [iii.] $\{0,a_2,\dots,a_e\}\subset Ap(S,a_1)$
	\item [iv.] $S$ is a MED-semigroup if and only if $Ap(S,a_1)=\{0,a_2,\dots,a_e\}$.	
\end{enumerate} 
\end{lemma}
\begin{lemma} \label{lemma3}[\cite{Ref31}, Proposition 5] Let $S_1$ and $S_2$ be two saturated numerical semigroups. Then $S_1\cap S_2$ is a saturated numerical semigroup.
\end{lemma}
Given $A$ a nonempty subset of $\mathbf{N}$ such that $gcd(A)=1$.  Then every saturated numerical semigroup containing $A$ must also
contain $\langle A\rangle$, and thus there are finitely many of them. We denote by $Sat(A)$ the intersection of all saturated numerical semigroups containing $A$. Thus,  we have that $Sat(A)$ is the smallest saturated semigroup containing $A$.

If $S$ is a saturated numerical semigroup and $A$ is a subset of $\mathbf{N}$  such that $Sat(A)=S$, then we will say that $A$ is a SAT-system of generators of $S$. We say that $A$ is a minimal SAT-system of generators of $S$ if in addition no proper subset of $A$ is a SAT-system of generators of $S$.
\begin{lemma} \label{thm 4}[\cite{Ref31}, Theorem 6] 
	Let $n_1<n_2<\dots <n_r$  be positive integers such that $ gcd \left\lbrace n_1,n_2,\dots,n_r\right\rbrace=1 $. For every $i\in \left\lbrace 1,2,\dots,r\right\rbrace$, set  $ d_i=gcd \left\lbrace n_1,n_2,\dots,n_i\right\rbrace $ and for all $j\in \left\lbrace 1,2,\dots,r-1\right\rbrace$ define 
	$$t_j=\max \left\lbrace t\in\mathbf{N}:n_j+td_i<d_{j+1}\right\rbrace. $$ 
	Then
	\begin{eqnarray*}
	Sat ( n_1,n_2,\dots,n_r) & = & \{0, n_1, n_1+ d_1,\dots,n_1+ t_1d_1,n_2,n_2+d_2,\dots,n_2+t_2d_2,
		\nonumber\\
		&  &\dots,n_{r-1}, n_{r-1}+d_{r-1},\dots,n_{r-1}+t_{r-1}n_{r-1},n_r,n_r+1,\rightarrow\}
	\end{eqnarray*}	
	
\end{lemma}
\begin{lemma} \label{thm 5}[\cite{Ref31}, Theorem 11]  Let $S$ be a saturated numerical semigroup. Then $ \{n_1,n_2,\dots,n_r\}=\{n\in S\setminus \{0\}:d_S (n)\neq d_S (n^\star)\quad\textrm{for all} \quad n^\star<n,\quad n^\star\in S\} $ is the unique minimal SAT system of generators of S.
\end{lemma}
Let $S$ be a numerical semigroup with multiplicity $m$ and conductor $c$. Note that since every nonnegative multiple of $m$ is an element of $S$, $c-1\notin S$. Thus, $c \not\equiv 1 \pmod m$.
\begin{theorem}\label{thm1}
	Let $S$ be a numerical semigroup. $S$ is a saturated numerical semigroup with multiplicity $p$ (prime) and conductor $c$ if and only if $S$ is one of the following:
	\begin{enumerate}
		\item [i.] If $c\equiv0 \pmod p$, then  $\langle p,c+1,c+2,\dots,c+p-1\rangle$,
		\item [ii.]	If $c\equiv i \pmod p$, then  $\langle p,c,c+1,\dots,c+p-i-1,c+p-i+1,\dots,c+p-1\rangle$ for  $i\in \left\lbrace 2,3,\dots,p-1\right\rbrace$.
	\end{enumerate}
\end{theorem}
\begin{proof} $\left( \Rightarrow\right)  $ \begin{enumerate}
		\item [i.] Let $S$ be the following numerical semigroup with multiplicity $p$ (prime) and
		conductor $c$, $c\equiv0\pmod p$:
		$$S=\langle p,c+1,c+2,\dots,c+p-1\rangle$$
		If $c\equiv0\pmod p$, then $p\mid c$. Therefore, $c=kp$  for some $k$. Thus,
		$$S=\langle p,c+1,c+2,\dots,c+p-1\rangle=\left\lbrace 0,p,2p,\dots,(k-1)p,kp,\rightarrow \right\rbrace$$ 
		(here $\rightarrow$ denotes that all integers larger than $kp$ are in the semigroup; we are denoting in this
		way that the conductor of $S$ is $kp$).
		
		If  $a\leq c$, then $a=rp$ for some $r$. For $a\in S\setminus \{0\}$
		$$ d_S (a)=gcd \left\lbrace x\in S:x\leq a\right\rbrace =p$$ 	
		and
		$$a+ d_S (a)=rp+p=\left(r+1 \right)p\in S.$$
		If $a>c$, then
		$$ d_S (a)=gcd \left\lbrace x\in S:x\leq a\right\rbrace =1,$$
		$a+ d_S (a)=a+1>c$ and $a+1\in S$.
		So, $S$ is a saturated numerical semigroup.
		\item [ii.] Let $S$ be the following numerical semigroup with multiplicity $p$ (prime) and
		conductor $c$, $c\equiv i\pmod p$ and $i\in \left\lbrace 2,3,\dots,p-1\right\rbrace$:
		$$S=\langle p,c,c+1,\dots,c+p-i-1,c+p-i+1,\dots,c+p-1\rangle.$$
		If $c\equiv i\pmod p$, then $p\mid \left(c-i\right)$. Therefore, $c=kp+i$ for some $k$. Thus,
		
		\begin{eqnarray*}
			S & = & \langle p,c,c+1,\dots,c+p-i-1,c+p-i+1,\dots,c+p-1\rangle
			\nonumber\\
			& = &\left\lbrace 0,p,2p,\dots,kp,kp+i\rightarrow \right\rbrace.
		\end{eqnarray*}
		If $a<c$, then $a=tp$ for some $t$. For $a\in S\setminus \{0\}$
		$$ d_S (a)=gcd \left\lbrace x\in S:x\leq a\right\rbrace =p$$ 
		and
		$$a+ d_S (a)=tp+p=\left(t+1 \right)p\in S.$$
		If $a\geq c$, then
		$$ d_S (a)=gcd \left\lbrace x\in S:x\leq a\right\rbrace =1,$$
		$a+ d_S (a)=a+1>c$ and $a+1\in S$.
		So, $S$ is a saturated numerical semigroup.
	\end{enumerate}
	$\left( \Leftarrow\right)  $ 
	Let $S$ be a saturated numerical semigroup with multiplicity $p$ (prime) and conductor $c$. According to Theorem \ref{thm 5}, $\{p=n_1,n_2,\dots,n_r\}=\{n\in S\setminus \{0\}:d_S (n)\neq d_S (n^\star)\quad\textrm{for all} \quad n^\star<n,\quad n^\star\in S\} $ is the unique minimal SAT system of generators of S. Since $p$ is a prime integer, the minimal SAT system of generators of S is $\{p=n_1,n_r\}$ or $\{p=n_1,n_r+1\}$. 
	\begin{enumerate}
		\item [i.] If the minimal SAT system of generators of S is $\{p=n_1,n_r\}$, then $n_r=kp+i$ for some $k$ and $i\in \left\lbrace 1,\dots,p-1\right\rbrace$. From Theorem \ref{thm 4}, 
		$$t_1=\max \left\lbrace t\in\mathbf{N}:p+tp<kp+i\right\rbrace=k-1$$ 
		is calculated and obtained as
		\begin{eqnarray*}
			Sat ( p=n_1,n_r) & = & \{0, p, p+ p,\dots,p+ (k-1)p,n_r,\rightarrow\}
			\nonumber\\
			& = &\{0, p,2p,\dots,kp,kp+i,\rightarrow\}.
		\end{eqnarray*}	
		So $c=kp+i$ for some $k$ and $i\in \left\lbrace 2,3,\dots,p-1\right\rbrace$, in other words $c\equiv i \pmod p$
		$S=\{0, p,2p,\dots,kp,kp+i,\rightarrow\}=\langle p,c,c+1,\dots,c+p-i-1,c+p-i+1,\dots,c+p-1\rangle$ 
		\item [ii.] If the minimal SAT system of generators of S is $\{p=n_1,n_r+1\}$, then $n_r=kp$ for some $k$. From Theorem \ref{thm 4},  $t_1=\max \left\lbrace t\in\mathbf{N}:p+tp<kp+1\right\rbrace=k-1$ is calculated and obtained as
		\begin{eqnarray*}
			Sat ( p=n_1,n_r+1) & = & \{0, p, p+ p,\dots,p+ (k-1)p,n_r+1,\rightarrow\}
			\nonumber\\
			& = &\{0, p,2p,\dots,kp,\rightarrow\}.
		\end{eqnarray*}	
		So $c=kp$ for some $k$, in other words $c\equiv 0 \pmod p$
		$S=\{0, p,2p,\dots,kp,\rightarrow\}=\langle p,c+1,c+2,\dots,c+p-1\rangle$
	\end{enumerate}
	
\end{proof}		
It is clear that by Theorem \ref{thm1} we get the following corollary.
\begin{corollary} \label{corollary 1}
There is only one saturated numerical semigroups with multiplicity $p$  (prime) and conductor $c$.	
\end{corollary}
		\section{Catenary degree of saturated numerical semigroups}
	\label{sec:4}
Let $S=\langle a_1<a_2<\dots<a_r\rangle$ and $s\in S$. If $Z(s)$ has more than one $\Re-$ classes, then $s=w+ai$ with
$w\in Ap(S,a_1)\setminus\{0\}$ and $i\in \left\lbrace 2,3\dots,r\right\rbrace$ \cite{Ref24}.
\paragraph{}In this section, we will calculate the catenary degree of the saturated numerical semigroups given in the Theorem \ref{thm1} by using the properties of saturated numerical semigroups and those given above.
\begin{corollary} \label{corollary 2}(\cite{Ref7}, Corollary 3) Let $S$ be a numerical semigroup minimally generated by
$\left\lbrace a_1,a_2,\dots,a_r\right\rbrace$ and let $s\in S$. If $s$ is minimal in $S$ with the condition $C(s)=C(S)$, then $s=w+a_i$ with $w\in Ap(S,a_1)\setminus\{0\}$ and $i\in \left\lbrace 2,3\dots,r\right\rbrace$.
 \end{corollary}
\begin{theorem}\label{thm2}
	Let $S$ be a numerical semigroup. If $S$ is a saturated numerical semigroup with multiplicity $p$ (prime) and conductor $c\equiv0 \pmod p$, then
	$$C(S)=2h+1$$
	where $c=ph$ for some possitive integer $h$.
\end{theorem}
\begin{proof} If $S$ is the saturated numerical semigroup with multiplicity $p$ (prime) and
	conductor $c\equiv0 \pmod p$, then $S=\langle p,c+1,\dots,c+p-1\rangle$ where $c=ph$ for some positive integer $h$ from Theorem \ref{thm1}. Since $S$ is a saturated numerical semigroup, $Ap(S,p)=\left\lbrace 0,c+1,\dots,c+p-1\right\rbrace$ from Lemma \ref{lemma1}. Let $s\in S$ and led $a_j$ be a minimal generator of $S$, $s=w+a_j$ with $w\in Ap(S,p)\setminus\{0\}$ and $j\in \left\lbrace 2,3\dots,p\right\rbrace$. This implies that $s=a_j+a_k$ for $k$ and $j\in \left\lbrace 2,3\dots,p\right\rbrace$, because $S$  is a saturated numerical semigroup. Therefore, the presence of the following can be
$$a_j+a_k=2c+(j+k-2)$$
where $a_k=c+(k-1)$ and $a_k=c+(j-1)$ from the form in which $S$ is defined. Thus,
$$Ap(S,p)\setminus\{0\}+\left\lbrace c+1,\dots,c+p-1\right\rbrace=\left\lbrace 2c+2,\dots,2c+2(p-1)\right\rbrace$$

Let's consider the set of elements in the form $s=a_j+a_k$. We firstly prove that every $Z(s)$
has at least two $\Re-$ classes. Assume the contrary that there is only one $\Re-$ classes in $Z(s)$.
\begin{enumerate}
	\item [i.] Let $j=k$. Then $s=a_j+a_k=2aj$ and $2a_j\notin Ap(S,p)$. Also, $s-p=2a_j-p\in S$.	Thus, one of factorizations of $s$ is $\left( 0,\dots,0,\underbrace{2}_{j \quad th \quad  component},0,\dots,0\right)$. Note that the $j$ th component of the factorization is $2$  and the other components is $0$. On the other hand, 
	$$ s=a_j+a_k=2aj=2(c+(j-1))=c+(c+2(j-1))$$
	and  let's write $hp$ instead of $c$
	$$ s=2aj=hp+(hp+2(j-1)).$$
	Since $2\leq2(j-1)\leq2(p-1)$, we have two cases:
		\begin{itemize}
		\item[$\bullet$] If $2(j-1)<p$, then $2j-1\neq j$ and one of factorizations of $s$ is $\left(h, 0,\dots,0,\underbrace{1}_{ \left( 2j-1\right) \quad th \quad  component},0,\dots,0\right)$.
		\item[$\bullet$] If $2(j-1)>p$, then $ s=2aj=hp+(hp+2(j-1))=hp+(hp+pr_1+s_1)=(h+r_1)p+(hp+s_1)$ for some possitive integer $r_1$ and non-negative integer $s_1<p$. Where $r_1=1$  and $s_1<p-1$ due to the values of $i$ and $j$.	
	
	\begin{itemize}
			\item[Case 1.] If  $s_1=0$, then $1\neq j$ and one of factorizations of $s$ is $\left(h+1, 0,\dots,0\right)$.
			\item[Case 2.] If $s_1\neq0$, then $s_1+1\neq j$ and one of factorizations of $s$  is 
			$\left(h+1, 0,\dots,0,\underbrace{1}_{ \left( s_1+1\right) \quad th \quad  component},0,\dots,0\right)$.
\end{itemize}	
\end{itemize}		
	\item [ii.] Let $j\neq k$. Then $s=a_j+a_k$  and $a_j+a_k\notin Ap(S,p)$. Also, $s-p\in S$. Thus, one of
	factorizations of $s$ is \[ \left( 0,\dots,0,\underbrace{1}_{j\quad th \quad  component},0,\dots,0,\underbrace{1}_{k\quad th \quad  component},0,\dots,0\right)\]. Note that the $j$ th and $k$ th components of the factorization are $1$ and the other components are
$0$. We also have two cases
	\begin{itemize}
	\item[$\bullet$] If $s=a_j+a_k\equiv0 \pmod p$, then $s=a_j+a_k=2c+j+k-2=2hp+(j+k-2)$ and $j+k-2\equiv0 \pmod p$. Therefore $j+k-2=pr_1$ for some possitive integer $s_1$. Since $2<j+k-2<2p-4$, where $r_1=1$ due to the
	values of $i$ and $k$. Thus, one of factorizations of $s$ is $\left(2h, 0,\dots,0\right)$.
	\item[$\bullet$] If $s=a_j+a_k\equiv s_2\pmod p$, then $s=a_j+a_k=2c+j+k-2=2hp+(j+k-2)$ and $j+k-2\equiv s_2 \pmod p$. Therefore, $j+k-2=pr_2+s_2$ for some possitive integers $r_2$ and $s_2$. We have two cases	
	\begin{itemize}
		\item[Case 1.] If $j+k-2< p$, then $s=a_j+a_k=2c+j+k-2=2hp+s_2$ and one of factorizations of $s$ is $\left(h, 0,\dots,0,\underbrace{1}_{ s_2\quad th \quad  component},0,\dots,0\right)$.
		\item[Case 2.] If $j+k-2> p$, then $s=a_j+a_k=2c+j+k-2=2hp+s_2$ since $j+k-2<2p-4$. One of factorizations of $s$ is
		
		 $\left(h+1, 0,\dots,0,\underbrace{1}_{ \left( s_2+2\right) \quad th \quad  component},0,\dots,0\right)$.
	\end{itemize}	
\end{itemize}		
\end{enumerate}
It is known that every element in the semigroup which is involved in one of its minimal presentations has a set of factorizations with at least two $\Re-$ classes. According to the above, $Z(s)$ has at least two $\Re-$ classes. Namely, for every $ x=(x_1,\dots,x_p ),y=(y_1,\dots,y_p )$ in $Z(s)$ we can write $supp(x)\bigcap  supp(y )=\emptyset$. Thus, $ gcd(x,y) =(0,\dots,0)$. This in particular implies that $ dist(x,y) =\max\left\lbrace |x|,|y |\right\rbrace $. The catenary degree of $S$ is the maximum of the lengths of these factorizations. 
\paragraph{} Again according to the above, one easily deduces that the largest length of a factorization in $Z(s)$ is reached when $s=a_j+a_k\equiv0 \pmod p$ for $j\neq k$. The factorization of $s$ is $\left(2h+1, 0,\dots,0\right)$. Since the catenary degree is the length of this factorization, $C(S)=2h+1$ from the Corollary \ref{corollary 2}. 
\end{proof}	
\begin{theorem}\label{thm3}
	Let $S$ be a numerical semigroup. If $S$ is a saturated numerical semigroup with multiplicity $p$ (prime) and conductor $c\equiv i \pmod p$ for $i\in \left\lbrace 2,3,\dots,p-1\right\rbrace$, then
$$C(S)= \left\lbrace  \begin{array}{ll}
	2h+2 & \textrm{ if}  \quad i<\frac{p+2}{2},\\
2h+3 & \textrm{ if}\quad i>\frac{p+2}{2},
\end{array} \right.$$
	where $c=ph+i$ for some possitive integer $h$.
\end{theorem}
\begin{proof}
	If $S$ is the saturated numerical semigroup with multiplicity $p$ (prime) and
	conductor $c\equiv i \pmod p$ for $i\in \left\lbrace 2,3,\dots,p-1\right\rbrace$, then $S=\langle p,c,c+1,\dots,c+p-i-1,c+p-i+1,\dots,c+p-1\rangle$ where $c=ph+i$ for some positive integer $h$ from Theorem \ref{thm1}. Since  $S$  is a saturated numerical semigroup, $Ap(S,p)=\left\lbrace 0,c,c+1,\dots,c+p-i-1,c+p-i+1,\dots,c+p-1\right\rbrace$  from Lemma \ref{lemma1}. Let $s\in S$ and $a_j$ be a minimal generator of $S$, $s=w+a_j$ with $w\in Ap(S,p)\setminus\{0\}$ and $j\in \left\lbrace 2,3\dots,p\right\rbrace$. Then $s=a_j+a_k$ for $k,j\in \left\lbrace 2,3\dots,p\right\rbrace$, because $S$ is a saturated numerical semigroup. Therefore, the presence of the following can be easily seen:
	$$a_j+a_k= \left\lbrace  \begin{array}{ll}
	2c+(j+k)-4 & \textrm{ if}  \quad 2 \leq j,k \leq  p-i+1,\\
	2c+(j+k)-3 & \textrm{ if}\quad \left( 2 \leq  j \leq  p-i+1\quad \textrm{ and}\quad p-i+2 \leq  k \leq  p\right) \\
 & \quad\textrm{ or}\\
		& \left( 2 \leq  k \leq  p-i+1\quad \textrm{ and}\quad p-i+2 \leq  j \leq  p\right), \\
    2c+(j+k)-2 & \textrm{ if}\quad p-i+2 \leq  j,k \leq  p
	\end{array} \right.$$
where \[a_k= \left\lbrace  \begin{array}{ll}
c+k-2 & \textrm{ if}  \quad 2 \leq k \leq  p-i+1,\\
c+k-1 & \textrm{ if}\quad p-i+2 \leq k \leq  p,
\end{array} \right.\]  and \[a_j= \left\lbrace  \begin{array}{ll}
c+j-2 & \textrm{ if}  \quad 2 \leq j \leq  p-i+1,\\
c+j-1 & \textrm{ if}\quad p-i+2 \leq j \leq  p,
\end{array} \right.\] from the form in which $S$ is defined.	Let's consider the set of elements in the form $s=a_j+a_k$. We firstly prove that every $Z(s)$ has at least two $\Re-$  classes. Assume to the contrary that there is only one $\Re-$ classes in $Z(s)$.
\begin{enumerate}
	\item [i.]	Let $j=k$. Then $s=a_j+a_k=2a_j$ and $2a_j\notin Ap(S,p)$. Also, $s-p=2a_j-p\in S$. Thus, one of factorizations of $s$  is $\left( 0,\dots,0,\underbrace{2}_{j \quad th \quad  component},0,\dots,0\right)$. Note that the
	$j$ th component of the factorization is $2$ and the other components are $0$. On the
	other hand,
	
	$$s=a_j+a_k=2a_j=\left\lbrace  \begin{array}{ll}
	2c+2j-4 & \textrm{ if}  \quad 2 \leq j \leq  p-i+1,\\
2c+2j-2 & \textrm{ if}\quad p-i+2 \leq j \leq  p,
	\end{array} \right.$$
	and let's write $c=hp+i$ instead of $c$
$$s=2a_j=\left\lbrace  \begin{array}{ll}
hp+(hp+i)+(i+2j-4) & \textrm{ if}  \quad 2 \leq j \leq  p-i+1,\\
hp+(hp+i)+(i+2j-2) & \textrm{ if}\quad p-i+2 \leq j \leq  p.
\end{array} \right.$$
		\begin{itemize}
		\item[$\bullet$] If $2 \leq j \leq  p-i+1$, then we have three cases.	
		\begin{itemize}
			\item[Case 1.] If $2 \leq i+2j-4 \leq  p-i-1$, then $i+2j-2\neq j$ and one of
			factorizations of $s$ is $\left(h, 0,\dots,0,\underbrace{1}_{ (i+2j-2)\quad th \quad  component},0,\dots,0\right)$.
			\item[Case 2.] If $ p-i-1 \leq i+2j-4 \leq  p-1$, then $i+2j-3\neq j$ and one of
			factorizations of $s$ is $\left(h, 0,\dots,0,\underbrace{1}_{ (i+2j-3)\quad th \quad  component},0,\dots,0\right)$.
	     	\item[Case 3.] If $ i+2j-4 \geq p $, then $s=2a_j=hp+(hp+i)+(i+2j-4)=hp+(hp+i)+pr_1+s_1=(h+r_1)p+(hp+i)+s_1$ for some possitive integers $r_1$ and non-negative integer $s_1<p$. Since 
	     	 $ \max(2a_j)=2c+2p-2i-2=hp+(hp+i)+p+(p-i-2)$ for $2 \leq j \leq  p-i+1$, where 	$r_1=1$ and $s_1<p-i-2$ due to the values of $i$ and $j$. Since  $ 0\leq s_1\leq  p-i-2\leq p-i-1$ and $s_1+2\neq j$, one of factorizations of $s$ is $\left(h+1, 0,\dots,0,\underbrace{1}_{ (s_1+2)\quad th \quad  component},0,\dots,0\right)$.
	     	\end{itemize}
	\item[$\bullet$] If $p-i+2 \leq j \leq  p$, then $i+2j-2>p$. Thus, $s=2a_j=hp+(hp+i)+(i+2j-2)=hp+(hp+i)+pr_2+s_2=(h+r_2)p+(hp+i)+s_2$ for some positive integer $r_2$ and nonnegative integer $s_2$ with $s_2<p$. Since $ \max(2a_j)=2c+2p-2=hp+(hp+i)+(2p+i-2)$ for $p-i+2 \leq j \leq  p$, where $r_2=1$ or $r_2=2$ and $s_2<p$ due to the values of $i$ and $j$. We have two cases.
	\begin{itemize}
		\item[Case 1.] If $0 \leq s_2 \leq p-i+1$, then $s_2+2\neq j$ and one of factorizations of $s$ is $\left(h+r_2, 0,\dots,0,\underbrace{1}_{ (s_2+2)\quad th \quad  component},0,\dots,0\right)$.
		\item[Case 2.] If $p-i+2 \leq s_2 \leq p$, then $s_2+1\neq j$  and one of factorizations of $s$ is $\left(h+r_2, 0,\dots,0,\underbrace{1}_{ (s_2+1)\quad th \quad  component},0,\dots,0\right)$.
	\end{itemize}	
	\end{itemize}
when $s=a_j+a_k=2a_j$  for $j=k$, other factorizations of $s$ we get are the difference from
$\left( 0,\dots,0,\underbrace{2}_{j \quad th \quad  component},0,\dots,0\right)$. These factorizations and $\left( 0,\dots,0,\underbrace{2}_{j \quad th \quad  component},0,\dots,0\right)$ are different $\Re-$ classes in $Z(s)$. This, in particular, would mean that there is a factorization $\left( s_1,\dots,s_p\right)$ of $s$ different from  $\left( 0,\dots,0,\underbrace{2}_{j \quad th \quad  component},0,\dots,0\right)$ such that $supp(\left( s_1,\dots,s_p\right))\bigcap  supp(\left( 0,\dots,0,\underbrace{2}_{j \quad th \quad  component},0,\dots,0\right))=\emptyset$. This contradicts with our acceptance.		
	\item [ii.]	Let $j\neq k$. Then $s=a_j+a_k$ and $a_j+a_k\notin Ap(S,p)$. Also, $s-p\in S$. Thus, one of factorizations of  $s$ is \[\left( 0,\dots,0,\underbrace{1}_{j\quad th \quad  component},0,\dots,0,\underbrace{1}_{k\quad th \quad  component},0,\dots,0\right)\]. Note that the $j$ th and $k$ th components of the factorization are $1$ and the other components are $0$. We also have two cases.
	\begin{itemize}
		\item[$\bullet$] If $s=a_j+a_k\equiv 0 \pmod p$, then $a_j+a_k=pr_3$ for some possitive integer $r_3$. Since  $2c+2p-2i\leq a_j+a_k \leq 2c+(2p-3)$, where $r_3=2h+2$ or $r_3=2h+3$  due to the values of $i$ and $j$. Thus, one of factorizations of $s$ is $\left( r_3,0,\dots,0\right)$. Note that the first component of the factorization is $r_3$ and the other components are $0$.	
		\item[$\bullet$] If $s=a_j+a_k\equiv s_4 \pmod p$, then $a_j+a_k=pr_4+ s_4$ for some possitive integers $r_4$ and $s_4$. Since  $2c+1\leq a_j+a_k \leq 2c+(2p-3)$, where $2h \leq r_4 \leq 2h+2$ and  $1 \leq s_4 \leq p-1$ due to the values of $i$ and $j$. We can write $a_j+a_k=pr_4+ s_4=(r_4-h)p+(hp+i)+(s_4-i)$. Therefore, we have three cases.
		\begin{itemize}
			\item[Case 1.] If $0 \leq s_4 -i\leq p-i+1$, then $s_4-i+2\neq k,j$ and one of factorizations of $s$ is $\left(r_4 -h, 0,\dots,0,\underbrace{1}_{ (s_4-i+2)\quad th \quad  component},0,\dots,0 \right)$.
			\item[Case 2.] If $p-i +2\leq s_4 -i\leq p$, then $s_4-i+1\neq k,j$ and one of factorizations of $s$ is $\left(r_4 -h, 0,\dots,0,\underbrace{1}_{ (s_4-i+1)\quad th \quad  component},0,\dots,0\right)$.
			\item[Case 3.] If $s_4 -i\leq 0$, then $a_j+a_k=pr_4+ s_4=(r_4-h-1)p+(hp+i)+(p+s_4-i)$. Since $p-i +1\leq p+s_4 -i\leq p$ and $p+s_4-i+1\neq k,j$ one of factorizations of $s$ is $\left(r_4 -h-1, 0,\dots,0,\underbrace{1}_{ (p+s_4-i+1)\quad th \quad  component},0,\dots,0\right)$.	
		\end{itemize}	
	\end{itemize}
when $s=a_j+a_k$ for $j\neq k$ , other factorizations of $s$ we get are the difference from $\left( 0,\dots,0,\underbrace{1}_{j\quad th \quad  component},0,\dots,0,\underbrace{1}_{k\quad th \quad  component},0,\dots,0\right)$. These factorizations and \[\left( 0,\dots,0,\underbrace{1}_{j\quad th \quad  component},0,\dots,0,\underbrace{1}_{k\quad th \quad  component},0,\dots,0\right)\] are different $\Re-$ classes in $Z(s)$. This, in particular, would mean that there is a factorization $\left( s_1,\dots,s_p\right)$ of $s$ different from
 \[\left( 0,\dots,0,\underbrace{1}_{j\quad th \quad  component},0,\dots,0,\underbrace{1}_{k\quad th \quad  component},0,\dots,0\right)\] such
that \[supp(\left( s_1,\dots,s_p\right))\bigcap  supp( \left( 0,\dots,0,\underbrace{1}_{j\quad th \quad  component},0,\dots,0,\underbrace{1}_{k\quad th \quad  component},0,\dots,0 \right))=\emptyset.\] This contradicts our
acceptance.	
\end{enumerate}
\paragraph{} It is known that every element in the semigroup involved in one of its minimal presentations has a set of factorizations with at least two $\Re-$ classes. According to the above, $Z(s)$ has at least two $\Re-$ classes. Namely, for every $ x=(x_1,\dots,x_p ),y=(y_1,\dots,y_p )$ in $Z(s)$ we can write $supp(x)\bigcap  supp(y )=\emptyset$. Thus, $ gcd(x,y) =(0,\dots,0)$. This in particular implies that $ dist(x,y) =\max\left\lbrace |x|,|y |\right\rbrace $. Therefore, the catenary degree of $S$ is the maximum of the lengths of these factorizations. 
\paragraph{} Again according to the above, one easily deduces that the largest length of a factorization in $Z(s)$ is reached when $s=a_j+a_k\equiv0 \pmod p$ for $j\neq k$. Since $ p< c< c+1<\dots< c+p-i-1< c+p-i+1<\dots<c+p-1 $, the smallest $s$ that meets these conditions $\min(s)=\min(a_j+a_k)=(c+p-i-1)+(c+p-i+1)=2c+2p-2i$. But can we find another element $s$ larger than $2c+2p-2i$. Namely, there is an element $s$ in $Z(s)$  with $2c+2p-2i<s <2c+2(p-1)$? Since $s=a_j+a_k\equiv0 \pmod p$  for $j\neq k$, if there is, then $s=a_j+a_k=2c+2p-2i+pk<2c+2(p-1)$ for some possitive integers $k$. When we make the necessary cancellations, the inequality $i>\frac{pk+2}{2}$. Where $k=0$ or $k=1$ due to the values of $i$. Thus, we have two cases:
\begin{itemize}
	\item[Case 1.] If $i>\frac{pk+2}{2}$, then $\max(s)=\max(a_j+a_k)=2c+2p-2i+p=2hp+2i+3p-2i=2h+3$. And the factorization of $s$ is $\left(2h+3, 0,\dots,0\right)$. Since the catenary degree is the length of this factorization,  $C(S)=2h+3$ from the Corollary \ref{corollary 2}.
	\item[Case 2.] If other cases, namely  $i<\frac{p+2}{2}$, then $\max(s)=\max(a_j+a_k)=2c+2p-2i=2hp+2i+2p-2i=2h+2$. And the factorization of $s$ is $\left(2h+2, 0,\dots,0\right)$. Since the catenary degree is the length of this factorization,  $C(S)=2h+2$ from the Corollary \ref{corollary 2}.
\end{itemize}	
\end{proof}	

\begin{example}
Let $S$ be the saturated numerical with multiplicity  $5$ and conductor  $33$. Then $S=<5,33,34,36,37>=\left\lbrace 0,5,10,15,20,25,30,33,\rightarrow \right\rbrace $. Where $p=5$, $i=3$ and $h=6$. The set of
elements of the form $w+n$ with $n\in \left\lbrace 33,34,36,37\right\rbrace $ and $w\in Ap(S,5)\setminus\{0\}=\left\lbrace 33,34,36,37\right\rbrace$ is $\left\lbrace 66,67,68,69,70,71,72,73,74\right\rbrace$. The factorizations of these elements are the
following:
$$ Z(66)=\left\lbrace \left(0, 2,0,0,0\right), \left(6,0,0,1,0\right)\right\rbrace $$
$$ Z(67)=\left\lbrace \left(0, 1,1,0,0\right), \left(6,0,0,0,1\right)\right\rbrace $$
$$ Z(68)=\left\lbrace \left(0, 0,2,0,0\right), \left(7,1,0,0,0\right)\right\rbrace $$
$$ Z(69)=\left\lbrace \left(0, 1,0,1,0\right), \left(7,0,1,0,0\right)\right\rbrace $$
$$ Z(70)=\left\lbrace \left(14, 0,0,0,0\right), \left(0,1,0,0,1\right),\left(0,0,1,1,0\right) \right\rbrace $$
$$ Z(71)=\left\lbrace \left(0, 0,1,0,1\right), \left(7,0,0,1,0\right)\right\rbrace $$
$$ Z(72)=\left\lbrace \left(0, 0,0,2,0\right), \left(7,0,0,0,1\right)\right\rbrace $$
$$ Z(73)=\left\lbrace \left(0, 0,0,1,1\right), \left(8,1,0,0,0\right)\right\rbrace $$
$$ Z(74)=\left\lbrace \left(0, 0,0,0,2\right), \left(8,0,1,0,0\right)\right\rbrace $$

Each element of $Z(s)$ is in the different $\Re-$ classes. The catenary degree of $S$ is reached for in
$70$. The catenary degree of $S$ is $14$. Moreover, since $i=3<\frac{5+2}{2}=\frac{p+2}{2}$ and $h=6$, it can
easily be found that $C(S)=2h+2=(2\cdot6)+2=14$ from the Theorem \ref{thm3}
\end{example}

\end{document}